\documentclass[12pt,reqno]{amsart}
\usepackage{amsmath,amsopn,amssymb,amsthm,multicol, dsfont, bbm}
\usepackage{hyperref}
\usepackage{color}
\usepackage{graphicx}

\DeclareMathOperator{\Ad}{Ad}

\DeclareMathOperator{\Aut}{Aut}

\DeclareMathOperator{\Span}{span}

\newcommand{\fr}{\mathfrak}
\newcommand{\op}{\operatorname}
\newcommand{\al}{\alpha}

\newcommand{\bb}{\mathbb}

\DeclareMathOperator{\SO}{SO}

\DeclareMathOperator{\OO}{O}
\DeclareMathOperator{\Sp}{Sp}

\DeclareMathOperator{\U}{U}

 \newtheorem{lemma} {Lemma} [section]
\newtheorem{theorem}[lemma]{Theorem} 
\newtheorem{remark}[lemma] {Remark} 
\newtheorem{prop} [lemma]{Proposition}  
\newtheorem{definition}[lemma] {Definition}

\begin{document}

\title[Geodesic orbit metrics in a class of homogeneous bundles]{Geodesic orbit metrics in a class of homogeneous bundles over quaternionic Stiefel manifolds}

\author{Andreas Arvanitoyeorgos, Nikolaos Panagiotis Souris and Marina Statha}
\address{University of Patras, Department of Mathematics, GR-26500 Rion, Greece}
\email{arvanito@math.upatras.gr}
\address{University of Patras, Department of Mathematics, GR-26500 Rion, Greece}
\email{nsouris@upatras.gr  }
\address{University of Patras, Department of Mathematics, GR-26500 Rion, Greece}
\email{statha@math.upatras.gr} 
\medskip

\begin{abstract}  Geodesic orbit spaces (or g.o. spaces) are defined as those homogeneous Riemannian spaces $(M=G/H,g)$ whose geodesics are orbits of one-parameter subgroups of $G$.  The corresponding metric $g$ is called a geodesic orbit metric.  We study the geodesic orbit spaces of the form $(\Sp(n)/\Sp(n_1)\times \cdots \times \Sp(n_s), g)$,  with $0<n_1+\cdots +n_s\leq n$. Such spaces include spheres, quaternionic Stiefel manifolds, Grassmann manifolds and quaternionic flag manifolds.  The present work is a contribution to the study of g.o. spaces $(G/H,g)$ with $H$ semisimple.

\medskip
\noindent 2020 {\it Mathematics Subject Classification.} Primary 53C25; Secondary  53C30.

\medskip
\noindent {\it Keywords}: Homogeneous geodesic; geodesic vector; geodesic orbit space; isotropy representation;  generalized quaternionic  flag manifold; quaternionic Stiefel manifold
\end{abstract}

\maketitle
 

\section{Introduction}
\markboth{Andreas Arvanitoyeorgos Nikolaos Panagiotis Souris and Marina Statha}{Invariant geodesic orbit metrics on real flag manifolds}

Geodesic orbit spaces $(M=G/H,g)$ are defined by the simple property that any geodesic $\gamma$ has the form 
\begin{equation*}\gamma(t)=\exp(tX)\cdot o,\end{equation*}
 where $\exp$ is the exponential map on $G$, $o=\gamma(0)$ is a point in $M$ and $\cdot$ denotes the action of $G$ on $M$.  These spaces were initially considered in \cite{Ko-Va} and up to today they have been extensively studied within various geometric contexts, including the Riemannian (\cite{GoNi}), pseudo-Riemannian (\cite{CaZa}), Finsler (\cite{YaDe}) and affine (\cite{Du}) context.  The classification of g.o. spaces remains an open problem.
 
 There are diverse examples of g.o. spaces, including the classes of symmetric spaces, weakly symmetric spaces (\cite{Ber-Ko-Va}, \cite{Wo2}), isotropy irreducible spaces (\cite{Wo1}), $\delta$-homogeneous spaces (\cite{Be-Ni-1}), Clifford-Wolf homogeneous spaces (\cite{Be-Ni-2}) and  \emph{naturally reductive spaces}.
 Reviews about g.o. spaces up to 2017 can be found in \cite{Ar2} and in the introduction of \cite{Ni2}. Various results up to 2020 are included in the recently published book \cite{Be-Ni-3}.
 
 Determining the g.o. metrics among the $G$-invariant metrics on a space $G/H$ presents some challenges.  The main challenge lies in the fact that the space of $G$-invariant metrics may have complicated structure, depending on whether the \emph{isotropy representation} of $H$ on the tangent space $T_{o}(G/H)$ contains pairwise equivalent submodules. To remedy this obstruction, various simplification results for g.o. metrics have been established (e.g. \cite{Ni2}, \cite{So1}).  A general observation is that the existence and the form of the g.o. metrics on $G/H$ depends to a large extent on the structure of the tangent space $T_{o}(G/H)$ induced from the isotropy representation and on the Lie algebraic relations between the corresponding submodules (e.g. \cite{CheNiNi}). 
 
 When $G$ is compact semisimple, the classification of the g.o. spaces $(G/H,g)$ with $H$ abelian and $H$ simple has been obtained in the works \cite{So2} and \cite{CheNiNi} respectively. On the other hand, the classification of compact g.o. spaces $(G/H,g)$ with $H$ semisimple remains open, while no general results are known for this case.  
The present paper is a continuation of our  work \cite{ArSoSt} 
 towards a study of g.o. metrics on a general family of spaces $G/H$, such that the isotropy representation of all of its members has a similar description.
 In that work we studied the geodesic orbit metrics on the spaces $\SO(n)/\SO(n_1)\times \cdots \times \SO(n_s)$ and $\U(n)/\U(n_1)\times \cdots \times \U(n_s)$ with $0<n_1+\cdots +n_s\leq n$.
These are spaces $G/H$ where $G$ is a compact classical Lie group and $H$ is a diagonally embedded product of Lie groups of the same type as $G$. 

In the present paper we study the geodesic orbit metrics on the spaces
 $\Sp(n)/\Sp(n_1)\times \cdots \times \Sp(n_s)$  with $0<n_1+\cdots +n_s\leq n$. These spaces properly include the spheres $\mathbb{S}^{4n-1}=\Sp(n)/\Sp(n-1)$, the quaternionic Stiefel manifolds $\Sp(n)/\Sp(n-k)$, the Grassmann manifolds $\Sp(n)/\Sp(k)\times \Sp(n-k)$,  and the  quaternionic  flag manifolds $\Sp(n)/\Sp(n_1)\times \cdots \times \Sp(n_s)$, $n_1+\cdots +n_s=n$ (\cite{Ra}). 
 If $n_1+\cdots +n_s< n$, each of these spaces can be viewed as a total space over a quaternionic Stiefel manifold, with fiber a quaternionic  flag manifold, i.e.
\begin{equation*}
 \Sp(m)/\Sp(n_1)\times \cdots \times \Sp(n_s)\rightarrow \Sp(n)/\Sp(n_1)\times \cdots \times \Sp(n_s)\rightarrow \Sp(n)/\Sp(m),
\end{equation*}
 with $m=n_1+\cdots +n_s$.

Our main result is the following.

\begin{theorem}\label{main2}
Let $G/H$  be the space $Sp(n)/Sp(n_1)\times \cdots \times Sp(n_s)$, where $0<n_1+\cdots +n_s\leq n$.  If $G/H\neq Sp(n)/Sp(n-1)$ (i.e. if $n-(n_1+\cdots+n_s)\neq 1$ or $s>1$) then a $G$-invariant Riemannian metric on $G/H$ is geodesic orbit if and only if it is the standard metric induced from the Killing form on the Lie algebra $\fr{sp}(n)$ of $Sp(n)$.  

If $G/H=Sp(n)/Sp(n-1)$ (i.e. $s=1$ and $n-n_1=1$) then a $G$-invariant metric $g$ on $G/H$ is geodesic orbit if and only if $g=g_{\mu}$, $\mu>0$, where $g_{\mu}$ denotes a one-parameter family of deformations of the standard metric $g_1$, along the fiber $Sp(1)$ of the fibration $Sp(n)/Sp(n-1)\rightarrow Sp(n)/Sp(1)\times Sp(n-1)$.
\end{theorem}

We remark that the non standard geodesic orbit metric $g_{\mu}$ appears in  \cite{Nik0} and \cite{Ta}.  

\smallskip
 The paper is structured as follows: In Section \ref{prelsec} we present  preliminary facts about homogeneous spaces and some special properties of the isotropy representation, to be used in our study.  In Section \ref{g.o.sec} we state and prove some useful properties of g.o. spaces.  
 In Section \ref{Isotropy2}, we describe the isotropy representation of the spaces $G/H=Sp(n)/Sp(n_1)\times \cdots \times Sp(n_s)$ and the structure of the tangent space $T_o(G/H)$ induced from the isotropy representation. In Section \ref{proofPrep} we state and prove some preliminary propositions required for the proof of Theorem \ref{main2}.  Finally, in Section \ref{proof2} we prove Theorem \ref{main2}.

\medskip
\noindent 
{\bf Acknowledgement.} This research is co-financed by Greece and the European Union (European Social Fund- ESF) through the Operational Programme ``Human Resources Development, Education and Lifelong Learning 2014-2020" in the context of the project ``Geodesic orbit metrics on homogeneous spaces of classical Lie groups" (MIS 5047124).

\section{Preliminaries}\label{prelsec}

\subsection{Invariant metrics on homogeneous spaces}

Let $G/H$ be a homogeneous space with origin $o=eH$ and assume that $G$ is compact.  Let $\fr{g},\fr{h}$ be the Lie algebras of $G,H$ respectively.  Moreover, let $\op{Ad}:G\rightarrow \op{Aut}(\fr{g})$ and $\op{ad}:\fr{g}\rightarrow \op{End}(\fr{g})$ be the adjoint representations of $G$ and $\fr{g}$ respectively, where $\op{ad}(X)Y=[X,Y]$.  Since $G$ is compact, there exists an $\op{Ad}$-invariant (and hence $\op{ad}$ skew-symmetric) inner product $B$ on $\fr{g}$, which we henceforth fix.  In turn, we have a $B$-orthogonal \emph{reductive decomposition}

\begin{equation}\label{ReducDecompos}\fr{g}=\fr{h}\oplus \fr{m},\end{equation}

\noindent where the subspace $\fr{m}$ is $\op{Ad}(H)$-invariant (and $\op{ad}(\fr{h})$-invariant) and is naturally identified with the tangent space of $G/H$ at the origin.

 A Riemannian metric $g$ on $G/H$ is called $G$-invariant if for any $x\in G$ the left translations $\tau_x:G/H\rightarrow G/H$, $pH\mapsto (xp)H$, are isometries of $(G/H,g)$.  The $G$-invariant metrics are in one to one correspondence with $\op{Ad}(H)$-invariant inner products $\langle \ , \ \rangle$ on $\fr{m}$.  Moreover, any such product corresponds to a unique endomorphism $A:\fr{m}\rightarrow \fr{m}$, called the \emph{corresponding metric endomorphism}, that satisfies 
 
 \begin{equation}\label{MetEnd}\langle X,Y \rangle =B(AX,Y) \ \ \makebox{for all} \ \ X,Y\in \fr{m}.\end{equation} 
 
\noindent It follows from Equation \eqref{MetEnd} that the metric endomorphism $A$ is symmetric with respect $B$, positive definite and $\op{Ad}(H)$-equivariant, that is $(\op{Ad}(h)\circ A)(X)=(A\circ \op{Ad}(h))(X)$ for all $h\in H$ and $X\in \fr{m}$.  Conversely, any endomorphism on $\fr{m}$ with the above properties determines a unique $G$-invariant metric on $G/H$.  

Since $A$ is diagonalizable, there exists a decomposition $\fr{m}=\bigoplus_{j=1}^l\fr{m}_{\lambda_j}$ into eigenspaces $\fr{m}_{\lambda_j}$ of $A$, corresponding to distinct eigenvalues $\lambda_j$. Each eigenspace $\fr{m}_{\lambda_j}$ is $\op{Ad}(H)$-invariant.  When an $\op{Ad}$-invariant inner product $B$ and a $B$-orthogonal reductive decomposition \eqref{ReducDecompos} have been fixed, we will make no distinction between a $G$-invariant metric $g$ and its corresponding metric endomorphism $A$.

The form of the $G$-invariant metrics on $G/H$ depends on the \emph{isotropy representation} $\op{Ad}^{G/H}:H\rightarrow \op{Gl}(\fr{m})$, defined by $\op{Ad}^{G/H}(h)X:=(d\tau_h)_o(X)$, $h\in H$, $X\in \fr{m}$.  We consider a $B$-orthogonal decomposition

\begin{equation}\label{IsotropyDecom}\fr{m}=\fr{m}_1\oplus \cdots \oplus\fr{m}_s,\end{equation}

\noindent into $\op{Ad}^{G/H}$-invariant and irreducible submodules. We recall that two submodules $\fr{m}_i$ and $\fr{m}_j$ are equivalent if there exists an $\op{Ad}^{G/H}$-equivariant isomorphism $\phi:\fr{m}_i\rightarrow \fr{m}_j$. The simplest case occurs when all the submodules $\fr{m}_i$ are pairwise inequivalent.  Then any $G$-invariant metric $A$ on $G/H$ has a diagonal expression with respect to decomposition \eqref{IsotropyDecom}.  In particular, $\left.A\right|_{\fr{m}_j}=\lambda_j\op{Id}$, $j=1,\dots,s$.

The next proposition is useful to compute the isotropy representation of a reductive homogeneous space. 

\begin{prop}\textnormal{(\cite{Arv})}\label{isotrepr}
Let $G/H$ be a homogeneous space and let $\fr{g} = \fr{h}\oplus\fr{m}$ be a reductive decomposition of $\fr{g}$.  Let $h\in H$, $X\in \fr{h}$ and $Y\in\fr{m}$.  Then
$$
\Ad^{G}(h)(X + Y) = \Ad^{H}(h)X + \Ad^{G/H}(h)Y
$$
that is, the restriction $\Ad^{G}\big|_{H}$ splits into the sum $\Ad^{H}\oplus\Ad^{G/H}$.  We denote by $\chi$ the $\Ad^{G/H}$. 
\end{prop}

The following lemma provides a simple condition for proving that two $\Ad^{G/H}$-submodules are inequivalent.

\begin{lemma}\label{EquivalentLemma}Let $G/H$ be a homogeneous space with  reductive decomposition $\fr{g}=\fr{h}\oplus \fr{m}$ and let $\fr{m}_i,\fr{m}_j\subset \fr{m}$ be submodules of the isotropy representation $\Ad^{G/H}$.  Assume that for any pair of non-zero vectors $X\in \fr{m}_i$, $Y\in \fr{m}_j$ there exists a vector $a\in \fr{h}$ such that $[a,X]=0$ and $[a,Y]\neq 0$.  Then the submodules $\fr{m}_i,\fr{m}_j$ are $\Ad^{G/H}$-inequivalent.\end{lemma}

\begin{proof}
If $\fr{m}_i,\fr{m}_j$ are equivalent, then there exists an $\Ad^{G/H}$-equivariant isomorphism $\phi:\fr{m}_i\rightarrow \fr{m}_j$.  The $\operatorname{Ad}^{G/H}$-equivariance of $\phi$ implies that $\phi$ is $\op{ad}_{\fr{h}}$-equivariant.  In other words,

\begin{equation*}\phi([a,X])=[a,Y],\end{equation*}
\noindent for any $a\in \fr{h}$.  However, $\phi$ is an isomorphism, therefore, $[a,X]$ is non-zero if and only if $[a,Y]$ is non-zero, which contradicts the hypothesis of the lemma.  Hence, the submodules $\fr{m}_i,\fr{m}_j$ are inequivalent.
\end{proof}
\begin{remark} For any two $\Ad^{G/H}$-submodules $\fr{m}_1,\fr{m}_2$, we denote by $[\fr{m}_1,\fr{m}_2]$ the space generated by the vectors $[X_1,X_2]$ where $X_1\in \fr{m}_1$ and $X_2\in \fr{m}_2$.  Similarly, denote by $[\fr{h},\fr{m}_1]$ the space generated by the vectors $[a,X_1]$ where $a\in \fr{h}$ and $X_1\in \fr{m}_1$.  If $\fr{m}_1,\fr{m}_2$ are $B$-orthogonal then $[\fr{m}_1,\fr{m}_2]\subseteq \fr{m}$.  Indeed, $[\fr{m}_1,\fr{m}_2]$ is $B$-orthogonal to $\fr{h}$ because $B([\fr{m}_1,\fr{m}_2],\fr{h})\subseteq B(\fr{m}_1,[\fr{m}_2,\fr{h}])\subseteq B(\fr{m}_1,\fr{m}_2)=\{0\}$.  Moreover, by the Jacobi identity, $[\fr{m}_1,\fr{m}_2]$ is also an $\Ad^{G/H}$-submodule and $[\fr{h},\fr{m}_1]$ is an $\Ad^{G/H}$-submodule of $\fr{m}_1$.
\end{remark}

\section{Properties of geodesic orbit spaces}\label{g.o.sec}

\begin{definition}
A $G$-invariant metric $g$ on $G/H$ is called a geodesic orbit metric (g.o. metric) if any geodesic of $(G/H,g)$ through $o$ is an orbit of a one parameter subgroup of $G$.  Equivalently, $g$ is a geodesic orbit metric if for any geodesic $\gamma$ of $(G/H,g)$ though $o$ there exists a non-zero vector $X\in \fr{g}$ such that $\gamma(t)=\exp (tX)\cdot o$, $t\in \mathbb R$.  The space $(G/H,g)$ is called a geodesic orbit space (g.o. space). 

\end{definition}

Let $G/H$ be a homogeneous space with $G$ compact, fix an $\op{Ad}$-invariant inner product $B$ on $\fr{g}$ and consider the $B$-orthogonal reductive decomposition \eqref{ReducDecompos}.  Moreover, identify each $G$-invariant metric on $G/H$ with the corresponding metric endomorphism $A:\fr{m}\rightarrow \fr{m}$.

  We have the following condition.

\begin{prop}\emph{(\cite{AlAr}, \cite{So1})}\label{GOCond} The metric $A$ on $G/H$ is geodesic orbit if and only if for any vector $X\in\fr{m}\setminus \left\{ {0} \right\}$ there exists a vector $a\in \fr{h}$ such that  

\begin{equation}\label{cor}[a+X,AX]=0.\end{equation}
\end{prop}

The following result, which we will call the \emph{normalizer lemma}, can be used to simplify the necessary form of the g.o. metrics on $G/H$ by using the normalizer $N_G(H^0)$.

\begin{lemma}\label{NormalizerLemma}\emph{(\cite{Ni3})}
The inner product $\langle \ ,\ \rangle$, generating the metric of a geodesic orbit Riemannian space $(G/H,g)$, is not only $\op{Ad}(H)$-invariant but also $\op{Ad}(N_G(H^0))$-invariant, where $N_G(H^0)$ is the normalizer of the unit component $H^0$ of the group $H$ in $G$.\end{lemma} 

\noindent As a result of the normalizer lemma, the metric endomorphism $A$ of a g.o. metric on $G/H$ is $\op{Ad}(N_G(H^0))$-equivariant. We will now state a complementary result to the normalizer lemma for compact spaces, that characterizes the restriction of a g.o. metric to the compact Lie group $N_G(H^0)/H^0$. 

\begin{lemma}\label{DualNormalizer}
Let $(G/H,g)$ be a compact geodesic orbit space with the $B$-orthogonal reductive decomposition $\fr{g}=\fr{h}\oplus \fr{m}$, where $B$ is an $\op{Ad}$-invariant inner product on $\fr{g}$. Let $A:\fr{m}\rightarrow \fr{m}$ be the corresponding metric endomorphism of $g$, and let $\fr{n}\subseteq \fr{m}$ be the Lie algebra of the compact Lie group $N_G(H^0)/H^0$.  Then the restriction of $A$ to $\fr{n}$   defines a bi-invariant metric on $N_G(H^0)/H^0$.
\end{lemma}
 \begin{proof}  We denote by $\fr{n}_{\fr{g}}(\fr{h})\subset \fr{g}$ the Lie algebra of $N_G(H^0)$.  We have a $B$-orthogonal decomposition $\fr{n}_{\fr{g}}(\fr{h})=\fr{h}\oplus \fr{n}$, where $\fr{n}$ coincides with the Lie algebra of $N_G(H^0)/H^0$.  Moreover, we have a $B$-orthogonal decomposition $\fr{m}=\fr{n}\oplus \fr{p}$, where $\fr{p}$ coincides with the tangent space of $G/N_G(H^0)$ at the origin.  By the normalizer lemma, the restriction of $A$ on $\fr{p}$ defines an invariant metric on $G/N_G(H^0)$, and thus $A\fr{p}\subseteq \fr{p}$.  By taking into account the symmetry of $A$ with respect to the product $B$, we deduce that $B(A\fr{n},\fr{p})=B(\fr{n},A\fr{p})\subseteq B(\fr{n},\fr{p})=\{0\}$.  Hence, the image $A\fr{n}$ is $B$-orthogonal to $\fr{p}$ which, along with decomposition $\fr{m}=\fr{n}\oplus \fr{p}$, yields $A\fr{n}\subseteq \fr{n}$. Therefore, the restriction $\left.A\right|_{\fr{n}}:\fr{n}\rightarrow \fr{n}$ defines a left-invariant metric on $N_G(H^0)/H^0$.  Since $A$ is a g.o. metric on $G/H$, Proposition \ref{GOCond} implies that for any $X\in \fr{n}$ there exists a vector $a\in \fr{h}$ such that $0=[a+X,AX]=[a+X,\left.A\right|_{\fr{n}}X]$. Therefore, by the same proposition, $\left.A\right|_{\fr{n}}$ defines a g.o. metric on $N_G(H^0)/H^0$.  On the other hand, any left-invariant g.o. metric on a Lie group is necessarily bi-invariant (\cite{AlNi}), and hence $\left.A\right|_{\fr{n}}$ is a bi-invariant metric on $N_G(H^0)/H^0$.
 \end{proof}

\begin{remark}\label{conclusion} Let $\fr{p}\subset \fr{m}$ be the tangent space of $G/N_G(H^0)$, let $\fr{n}\subset \fr{m}$ be the Lie algebra of $N_G(H^0)/H^0$, and consider the decomposition $\fr{m}=\fr{n}\oplus \fr{p}$.  By combining lemmas \ref{NormalizerLemma} and \ref{DualNormalizer}, we conclude that the metric endomorphism $A$ corresponding to a g.o. metric has the block-diagonal form
\begin{equation*}A=\begin{pmatrix}\left.A\right|_{\fr{n}} & 0\\
0& \left.A\right|_{\fr{p}}\end{pmatrix},\end{equation*}

\noindent where $\left.A\right|_{\fr{n}}$ defines a bi-invariant metric on $N_G(H^0)/H^0$ and $\left.A\right|_{\fr{p}}$ defines a g.o. metric on $G/N_G(H^0)$.
\end{remark}

\smallskip
The following  lemma describes the invariant g.o. metrics on compact Lie groups in terms of their metric endomorphism with respect to an $\op{Ad}$-invariant inner product $B$.  Recall that the Lie algebra $\fr{g}$ of a compact Lie group $G$ has a direct sum decomposition $\fr{g}=\fr{g}_1\oplus \cdots \oplus \fr{g}_k\oplus \fr{z}$, where $\fr{g}_j$ are  simple ideals of $\fr{g}$ and $\fr{z}$ is its center.

\begin{lemma}\label{GOLieGroups}\emph{(\cite{AlNi}, \cite{So1})} 
Let $G$ be a compact Lie group with Lie algebra $\fr{g}=\fr{g}_1\oplus \cdots \oplus \fr{g}_k\oplus \fr{z}$.  A left invariant metric $A$ on $G$ is a g.o. metric if and only if it is bi-invariant.  In particular, $A$ is a g.o. metric if and only if 

\begin{equation*}A=\begin{pmatrix} 
 \lambda_1\left.\op{Id}\right|_{\fr{g}_1} & 0 & \cdots &0\\
  \vdots & \ddots & \cdots &\vdots\\
  0&\cdots &\lambda_k\left.\op{Id}\right|_{\fr{g}_k} &0\\
  0& \cdots & 0 & \left.A\right|_{\fr{z}}
  \end{pmatrix}, \ \lambda_j>0.
\end{equation*}
\end{lemma}

\smallskip
The following terminology will be useful.
Let $W$ be a subspace of a vector space $V$ and we write $V=W\oplus W^{\bot}$ with respect to some inner product $B$ on $V$.  Then, for $v\in V$ it is $v=w+w'$, where $w\in W$ and $w'\in W^{\bot}$.  

\begin{definition}
The vector $v$ 
  has \emph{non zero projection on $W$} if $w'\neq 0$. 
  A subset $S$ of $V$ has non zero projection on $W$ if there exists a vector $v\in S$ that has non zero projection on $W$.
  \end{definition}
  Finally, the folowing lemma is useful, since it will enable us to equate some of the eigenvalues of a g.o. metric.

\begin{lemma}\label{EigenEq}\emph{(\cite{So1})}
 Let $(G/H,g)$ be a g.o. space with $G$ compact and with corresponding metric endomorphism $A$ with respect to an $\op{Ad}$-invariant inner product $B$. Let $\fr{m}$ be the $B$-orthogonal complement of $\fr{h}$ in $\fr{g}$.\\ 
\textbf{1.} Assume that $\fr{m}_1,\fr{m}_2$ are $\operatorname{ad}(\fr{h})$-invariant, pairwise $B$-orthogonal subspaces of $\fr{m}$ such that $[\fr{m}_1,\fr{m}_2]$ has non zero projection on ${(\fr{m}_1\oplus \fr{m}_2)^\bot}$. Let $\lambda_1,\lambda_2$ be eigenvalues of $A$ such that $\left.A\right|_{\fr{m}_i}=\lambda_i\op{Id}$, $i=1,2$.  Then $\lambda_1=\lambda_2$.\\
\textbf{2.} Assume that $\fr{m}_1,\fr{m}_2,\fr{m}_3$ are $\operatorname{ad}(\fr{h})$-invariant, pairwise $B$-orthogonal subspaces of $\fr{m}$ such that $[\fr{m}_1,\fr{m}_2]$ has non-zero projection on $\fr{m}_3$. Let $\lambda_1,\lambda_2,\lambda_3$ be eigenvalues of $A$ such that $\left.A\right|_{\fr{m}_i}=\lambda_i\op{Id}$, $i=1,2,3$.  Then $\lambda_1=\lambda_2=\lambda_3$.
\end{lemma}

\section{The space $\Sp(n)/\Sp(n_1)\times\cdots\times\Sp(n_s)$, $\sum n_i\le n$}\label{Isotropy2}

We compute  the isotropy representation of $M=G/H=\Sp(n)/\Sp(n_1)\times\cdots\times\Sp(n_s)$, $n_1+\cdots +n_s\le n$.
Denote by  $\nu_{2n}$ the standard representation of  $\Sp(n)$, 
that is   $\nu_{2n}  : \Sp(n)\to \Aut(\bb{C}^{2n})$. 
 Then  the complexified adjoint representation of $\Sp(n)$ is given by  $\Ad^{\Sp(n)}\otimes\bb{C} = S^{2}\nu_{2n}$, where  $S^{2}$ is the second symmetric power of $\nu_{2n}$.
Let $\sigma _{n_i}: \Sp(n_{1})\times\cdots\times\Sp(n_{s})\to\Sp(n_{i})$ be the projection onto the $i$-factor and $\varphi_{i} = \nu_{2n_{i}}\circ\sigma_{n_{i}}$
 be the  projection of the standard representation of $H$, i.e.  
$$ 
\Sp(n_{1})\times\cdots\times\Sp(n_{s})\stackrel{\sigma_{n_{i}}}{\longrightarrow}\Sp(n_{i})\stackrel{\nu_{2n_{i}}}{\longrightarrow}\Aut(\bb{C}^{2n_{i}}). 
$$
We set $n_0 := n-(n_1+n_2+\cdots + n_s)$.

Then we have:
\begin{eqnarray}
\Ad^{G}\otimes\mathbb{C}\big|_{H} &=& S^2\nu _{2n}\big|_{H} 
   = S^2(\varphi_{1}\oplus\cdots\oplus \varphi_{s} \oplus \mathbbm{1}_{2n_0}) 
   = S^{2}\varphi_{{1}}\oplus S^{2}\varphi_{{2}}\oplus\cdots\oplus S^{2}\varphi_{s}\oplus S^{2}\mathbbm{1}_{2n_0}\nonumber \\
&&\oplus\{(\varphi_{{1}}\otimes \varphi_{{2}})\oplus\cdots\oplus(\varphi_{{1}}\otimes \varphi_{{s}})\}\oplus\{(\varphi_{{2}}\otimes \varphi_{{3}})\oplus\cdots\oplus (\varphi_2\otimes \varphi_{s})\} \oplus \cdots\oplus (\varphi_{s-1}\otimes \varphi_s)\nonumber \\
&&
\oplus(\varphi_{{1}}\otimes \mathbbm{1}_{2n_0})\oplus(\varphi_{{2}}\otimes \mathbbm{1}_{2n_0})\oplus\cdots\oplus(\varphi_{{s}}\otimes \mathbbm{1}_{2n_0})\label{FirstIsotr},
\end{eqnarray}
where $S^2\mathbbm{1}_{2n_0} = \mathbbm{1}\oplus\cdots\oplus \mathbbm{1}$  is the sum of $\binom{2n_0+1}{2}=n_0(2n_0+1)$ trivial representations. 


 The representation $S^{2}\varphi_{{1}}\oplus S^{2}\varphi_{{2}}\oplus\cdots\oplus S^{2}\varphi_{s}$ is equal to $\Ad ^H\otimes\mathbb{C}$, the complexified adjoint 
representation of  $H = \Sp(n_{1})\times\cdots\times\Sp(n_{s})$.
 
Then, in view of Equation \eqref{FirstIsotr}, Proposition \ref{isotrepr} implies that
the complexified isotropy representation of $G/H$ is given by
\begin{eqnarray}
\chi\otimes\bb{C} &=& S^{2}\mathbbm{1}_{2n_0}\oplus\{(\varphi_{{1}}\otimes \varphi_{{2}})\oplus\cdots\oplus(\varphi_{{1}}\otimes \varphi_{{s}})\}\oplus\{(\varphi_{{2}}\otimes \varphi_{{3}})\oplus\cdots\oplus (\varphi_2\otimes \varphi_{s})\} \oplus \cdots \oplus (\varphi_{s-1}\otimes \varphi_s) \nonumber\\
&&\oplus(\varphi_{{1}}\otimes \mathbbm{1}_{2n_0})\oplus(\varphi_{{2}}\otimes \mathbbm{1}_{2n_0})\oplus\cdots\oplus(\varphi_{{s}}\otimes \mathbbm{1}_{2n_0})\nonumber \\
&=&S^{2}\mathbbm{1}_{2n_0}\bigoplus_{1\leq i<j\leq s}{(\varphi_{{i}}\otimes \varphi_{{j}})}\bigoplus_{j=1}^s{(\varphi_{{j}}\otimes \mathbbm{1}_{2n_0})}.\label{isotropySp}
\end{eqnarray}
Also, $\dim_{\bb{C}}({\varphi_i\otimes \varphi_j}) = 4n_i n_j$  and each of $\varphi_j\otimes \mathbbm{1}_{2n_0}\cong(\varphi_j\otimes\mathbbm{1})\oplus\cdots\oplus(\varphi_j\otimes\mathbbm{1})$ is a sum of $2n_0$ equivalent representations, each of dimension $2n_j$, $j = 1,2,\ldots, s$ (and hence $\dim_{\bb{C}}({\varphi_j\otimes \mathbbm{1}_{2n_0}})=4kn_j$).
One can confirm that 
$$
\dim (\chi\otimes\bb{C})=n_0(2n_0+1)+\sum_{1\le i<j\le s}4n_in_j+\sum_{j=1}^s 4n_0n_j=2n^2-2\sum_{j=1}^s n_j^2+n_0=
\dim M.
$$

Expression  (\ref{isotropySp}) induces a decomposition of the complexified tangent space $\fr{m}$ of $G/H$ as
$$
\fr{m}\otimes\mathbb{C}=\fr{n}_1'\oplus\cdots\oplus\fr{n}_{n_0(2n_0+1)}' \bigoplus_{1\le i<j\le s}\fr{n}_{ij}\bigoplus_{j=1}^{s}\fr{n}_{0j},
$$
where $\dim_{\mathbb{C}}(\fr{n}_i') = 1$, $\dim_{\mathbb{C}}\fr{n}_{ij}=4n_in_j$ and $\fr{n}_{0j} = \fr{n}_1^{j}\oplus\fr{n}_2^{j}\oplus\cdots\oplus\fr{n}_{2n_0}^{j}$, with $\fr{n}_{\al}^{j}\cong\fr{n}_{\beta}^{j}$, $\al \neq \beta$ and $\dim_{\mathbb{C}}(\fr{n}^{j}_{\ell})= 2n_j$, $\ell=1,2,\ldots, 2n_0$.

The (real) decomposition of the tangent space of $G/H$ is given by
\begin{equation}\label{isotrsp}
\fr{m}=\fr{n}_1\oplus\cdots\oplus\fr{n}_{n_0(2n_0+1)} \bigoplus_{1\le i<j\le s}\fr{m}_{ij}\bigoplus_{j=1}^{s}\fr{m}_{0j},
\end{equation}
\noindent
where  $\fr{m}_{ij}\otimes\mathbb{C}=\fr{n}_{ij}$, $\fr{m}_{0j}\otimes\mathbb{C}=\fr{n}_{0j}$,
$\dim_{\mathbb{R}}(\fr{n}_i) = 1$, 
$\dim_{\mathbb{R}}(\fr{m}_{ij}) = 4n_in_j$,
$\dim_{\mathbb{R}}(\fr{m}_{0j}) =4n_0n_j$.
Also, $\fr{m}_{0j} =\fr{m}_1^j\oplus\cdots\oplus\fr{m}_{2n_0}^j$ with
$\fr{m}_\ell^j\otimes\mathbb{C}=\fr{n}_\ell^j$, $\fr{m}_\al^j\cong\fr{m}_\beta^j$, $\al\ne\beta$ and 
$\dim_{\mathbb{R}}(\fr{m}_\ell^j) = 2n_j$, $\ell=1, \dots, 2n_0$.
Note that $\fr{n}_1\oplus\cdots\oplus\fr{n}_{n_0(2n_0+1)}\cong\fr{sp}(n_0)$.

\begin{remark} If $n_0=0$ the isotropy representation of $G/H$ has no equivalent representations and all above expressions simplify.
\end{remark}

\medskip
We now give explicit matrix representations  of $\fr{m}_{ij}, \fr{m}_{0j}$.
Recall the Lie algebra of $\Sp(n)$,
\begin{equation*}
\mathfrak{sp}(n)=\left\{\begin{pmatrix}
X & -{}\bar{Y}\\
Y & \bar{X}
\end{pmatrix}\ \Big\vert 
 \begin{array}{l}X\in\mathfrak{u}(n), \ 
Y\ \mbox{ is  a}\  n\times n\ \mbox{complex symmetric matrix}
\end{array}
\right\}\subset\fr{u}(2n).\ 
\end{equation*} 
For $i=1, \dots , s$, we embed 
$\fr{sp}(n_i)=\left\{\begin{pmatrix}
X_i & -{}\bar{Y}_i\\
Y_i & \bar{X}_i
\end{pmatrix}\right\}$ 
in $\fr{sp}(n)$  as

\begin{equation*}
 \left\{\left( 
 \begin{array}{ccccccccccc}
  0 & \cdots &  0  & \cdots & 0 & \vline & 0 & \cdots &0 &\cdots &0\\ 
   \vdots &  &  \vdots  &  & \vdots & \vline & \vdots &  &\vdots &  &0\\
 0 & \cdots &  X_i  & \cdots & 0  & \vline& 0 &\cdots &-{}\bar{Y}_i &\cdots &0\\
  \vdots &  &  \vdots  &  & \vdots & \vline & \vdots &  &\vdots &  &0\\
0 & \cdots &  0  & \cdots & 0  & \vline& 0 &\cdots &0 &\cdots &0\\
\hline
0 & \cdots &  0  & \cdots & 0 & \vline & 0 & \cdots &0 &\cdots &0\\ 
 \vdots &  &  \vdots  &  & \vdots & \vline & \vdots &  &\vdots &  &0\\
0 & \cdots &  Y_i  & \cdots & 0  & \vline& 0 &\cdots &\bar{X}_i &\cdots &0\\
 \vdots &  &  \vdots  &  & \vdots & \vline & \vdots &  &\vdots &  &0\\
0 & \cdots &  0  & \cdots & 0  & \vline& 0 &\cdots &0 &\cdots &0
  \end{array}
  \right)
  \right\}.
 \end{equation*}

We consider the $\Ad(\Sp(n))$-invariant inner product $B:\fr{sp}(n)\times\fr{sp}(n)\to\mathbb{R}$, given by
\begin{equation}\label{KillSp(n)}
B(X, Y)= -{\rm Trace}(XY), \quad X, Y\in\fr{sp}(n),
\end{equation}
\bigskip
and we obtain 
a $B$-orthogonal decomposition  $\fr{g}=\fr{h}\oplus\fr{m}$, where $\fr{h}=\fr{sp}(n_1)\oplus\cdots\oplus\fr{sp}(n_s)$ and $\fr{m}\cong T_o(G/H)$.  
 We note that $B$ is a multiple of the Killing form of $\fr{sp}(n)$.

Next, we consider a basis for $\fr{g}=\fr{sp}(n)$ as follows.
Let $M_{2n}\mathbb{C}$ be the set of $2n\times 2n$ complex matrices and we consider the following matrices in  $M_{2n}\mathbb{C}$ with zeros in all entries except the ones indicated:

$E_{ab}$ with $1$ in $(a, b)$-entry and $1$ in $(n+a, n+b)$-entry.

$F_{ab}$ with $i$ in $(a, b)$-entry and $-i$ in $(n+a, n+b)$-entry.

$G_{ab}$ with $-1$ in $(a, n+b)$-entry and $1$ in $(n+b, a)$-entry.

$H_{ab}$ with $i$ in $(a, n+b)$-entry and $i$ in $(n+b, a)$-entry.

For $1\le a <b\le 2n$ we set
\begin{equation}\label{basissp}
e_{ab}=E_{ab}-E_{ba}, \quad f_{ab}=F_{ab}+F_{ba}, \quad g_{ab}=G_{ab}+G_{ba}, \quad
h_{ab}=H_{ab}+H_{ba}.   
\end{equation}
Then the set $\mathcal{B}=\{e_{ab}, f_{ab}, g_{ab}, h_{ab}: 1\le a<b\le n; \ f_{aa}, g_{aa}, h_{aa}: 1\le a\le n\}$ 
constitutes a basis of $\fr{sp}(n)$, which is orthogonal with respect to $B$.

The above matrices have the form
{\small
\begin{equation*}
 e_{ab}=\left( 
 \begin{array}{ccccccccccc}
  \square &  &   &  &  & \vline & \square &  &  &  &  \\ 
   & \ddots &  &  &  & \vline &  & \ddots &  &  & \\
   &  & \square  & 1 &  & \vline &  &  & \square &  & \\
   &  & -1  & \ddots &  & \vline &  &  &  & \ddots & \\
   &  &   &  & \square & \vline &  &  &  &  & \square\\
  \hline
   \square &  &   &  &  & \vline & \square &  &  &  &  \\
    & \ddots &  &  &  & \vline &  & \ddots &  &  & \\
     &  & \square  &  &  & \vline &  &  & \square & 1 & \\
       &  &   & \ddots &  & \vline &  &  & -1 & \ddots & \\
        &  &   &  & \square & \vline &  &  &  &  & \square
  \end{array}
  \right),  
  f_{ab}=\left( 
 \begin{array}{ccccccccccc}
  \square &  &   &  &  & \vline & \square &  &  &  &  \\ 
   & \ddots &  &  &  & \vline &  & \ddots &  &  & \\
   &  & \square  & i &  & \vline &  &  & \square &  & \\
   &  & i  & \ddots &  & \vline &  &  &  & \ddots  & \\
   &  &   &  & \square & \vline &  &  &  &  & \square\\
  \hline
   \square &  &   &  &  & \vline & \square &  &  &  &  \\
    & \ddots &  &  &  & \vline &  & \ddots &  &  & \\
     &  & \square  &  &  & \vline &  &  & \square & -i & \\
       &  &   & \ddots &  & \vline &  &  & -i & \ddots & \\
        &  &   &  & \square & \vline &  &  &  &  & \square
  \end{array}
  \right), 
 \end{equation*} 
 }
 
 {\small
 \begin{equation*}
 g_{ab}=\left( 
 \begin{array}{ccccccccccc}
  \square &  &   &  &  & \vline & \square &  &  &  &  \\ 
   & \ddots &  &  &  & \vline &  & \ddots &  &  & \\
   &  & \square  &  &  & \vline &  &  & \square & -1 & \\
   &  &   & \ddots &  & \vline &  &  & -1 &\ddots  & \\
   &  &   &  & \square & \vline &  &  &  &  & \square\\
  \hline
   \square &  &   &  &  & \vline & \square &  &  &  &  \\
    & \ddots &  &  &  & \vline &  & \ddots &  &  & \\
     &  & \square  & 1 &  & \vline &  &  & \square &  & \\
       &  & 1  & \ddots &  & \vline &  &  &  & \ddots & \\
        &  &   &  & \square & \vline &  &  &  &  & \square
  \end{array}
  \right),  
   h_{ab}=\left( 
 \begin{array}{ccccccccccc}
  \square &  &   &  &  & \vline & \square &  &  &  &  \\ 
   & \ddots &  &  &  & \vline &  & \ddots &  &  & \\
   &  & \square  &  &  & \vline &  &  & \square & i & \\
   &  &   & \ddots &  & \vline &  &  & i &\ddots  & \\
   &  &   &  & \square & \vline &  &  &  &  & \square\\
  \hline
   \square &  &   &  &  & \vline & \square &  &  &  &  \\
    & \ddots &  &  &  & \vline &  & \ddots &  &  & \\
     &  & \square  & i &  & \vline &  &  & \square &  & \\
       &  & i  & \ddots &  & \vline &  &  &  & \ddots & \\
        &  &   &  & \square & \vline &  &  &  &  & \square
  \end{array}
  \right).  
 \end{equation*}

 

Recall also the relations $e_{ba}=-e_{ab}$, $f_{ba}=f_{ab}$, $g_{ba}=g_{ab}$ and $h_{ba}=h_{ab}$.  The next lemma follows from straightforward calculations.

\begin{lemma}\label{rel2}
The Lie-bracket relations among the vectors {\rm (\ref{basissp})} are given as follows: 
\begin{align*}
[e_{ij},e_{lm}]&=\delta_{jl}e_{im}+\delta_{im}e_{jl}-\delta_{il}e_{jm}-\delta_{jm}e_{il} &[e_{ij},f_{lm}]&=\delta_{jl}f_{im}-\delta_{im}f_{jl}-\delta_{il}f_{jm}+\delta_{jm}f_{il}\\
[e_{ij},g_{lm}]&=-\delta_{jl}g_{im}+\delta_{im}g_{jl}-\delta_{il}g_{jm}+\delta_{jm}g_{il} &[e_{ij},h_{lm}]&=-\delta_{jl}h_{im}+\delta_{im}h_{jl}-\delta_{il}h_{jm}+\delta_{jm}h_{il}\\
[f_{ij},f_{lm}]&=-\delta_{jl}e_{im}-\delta_{im}e_{jl}-\delta_{il}e_{jm}-\delta_{jm}e_{il} &[f_{ij},g_{lm}]&=-\delta_{jl}h_{im}-\delta_{im}h_{jl}-\delta_{il}h_{jm}-\delta_{jm}h_{il}\\
[f_{ij},h_{lm}]&=\delta_{jl}g_{im}+\delta_{im}g_{jl}+\delta_{il}g_{jm}+\delta_{jm}g_{il} &[g_{ij},g_{lm}]&=-\delta_{jl}e_{im}-\delta_{im}e_{jl}-\delta_{il}e_{jm}-\delta_{jm}e_{il} \\
[g_{ij},h_{lm}]&=-\delta_{jl}f_{im}-\delta_{im}f_{jl}-\delta_{il}f_{jm}-\delta_{jm}f_{il} &[h_{ij},h_{lm}]&=-\delta_{jl}e_{im}-\delta_{im}e_{jl}-\delta_{il}e_{jm}-\delta_{jm}e_{il}.
\end{align*}
\end{lemma}

 A choice for the modules in the decomposition (\ref{isotrsp}) is the following:
\begin{eqnarray*}
\fr{m}_{ij}&=&\Span_{\mathbb{R}}\{e_{ab}, f_{ab},g_{ab},h_{ab}:  n_0+n_1+\cdots +n_{i-1}+1\le a\le 
n_0+n_1+\cdots +n_i, \\
 && \ \ \ \ \ \ \ \ \ \quad  \ \quad \quad  n_0+n_1+\cdots +n_{j-1}+1\le b\le n_0+n_1+\cdots +n_j\}, \ 1\le i<j\le s\\
\fr{m}_{0j}&=&\Span_{\mathbb{R}}\{e_{ab}, f_{ab},g_{ab},h_{ab}: 1\le a\le n_0,\ n_0+n_1+\cdots +n_{j-1}+1\le b\le 
n_0+n_1+\cdots +n_j\}, \ j=1,\dots ,s.
\end{eqnarray*}

The $n_0(2n_0+1)$ trivial representations in (\ref{isotrsp}) generate the Lie algebra
$$\fr{n} = \Span_{\mathbb{R}}\{e_{ab}, f_{ab}, g_{ab}, h_{ab}: 1\le a<b\le n_0; \ 
f_{aa}, g_{aa}, h_{aa}: 1\le a\le n_0\},
$$
isomorphic to $\fr{sp}(n_0)$.

The equivalent modules in the decomposition of $\fr{m}_{0j}$ are given by
\begin{eqnarray*}
\fr{m}^j_\ell =\Span_{\mathbb{R}}\{e_{\ell b}, f_{\ell b}, g_{\ell b}, h_{\ell b}: n_0+n_1+\cdots +n_{j-1}+1\le b\le 
n_0+n_1+\cdots +n_j\}, \ \ell = 1, \dots , 2k.
\end{eqnarray*}
Also, for $j =1, \dots , s$, we have 
\begin{eqnarray*}
\fr{sp}(n_j)&=&\Span_{\mathbb{R}}\{e_{ab}, f_{ab}, g_{ab}, h_{ab}:  n_0+n_1+\cdots +n_{j-1}+1\le a<b\le 
n_0+n_1+\cdots +n_j; \\
 && \ \ \ \ \ \ \ \ f_{aa}, g_{aa}, h_{aa}: n_0+n_1+\cdots +n_{j-1}+1\le a\le 
n_0+n_1+\cdots +n_j
 \}.
\end{eqnarray*}

In summary, we obtain the $B$-orthogonal decomposition
\begin{equation}\label{dd2}\fr{m}=\fr{n}\oplus \fr{p},
\end{equation}
where
$$\fr{n}\cong\fr{sp}(n_0), \quad \fr{p}=\bigoplus_{1\le i<j\le s}\fr{m}_{ij}\bigoplus_{j=1}^{s}\fr{m}_{0j}=\bigoplus_{0\le i<j\le s}\fr{m}_{ij}.
$$
Decomposition (\ref{dd2}) as a subspace of $\fr{sp}(n)$ can be depicted in the following $2n\times 2n$ complex matrix, where the positions represented by  $\square$ are zero matrices (corresponding to positions of $\fr{sp}(n_i)$, $i=1, \dots s$):

\begin{equation*}
 \left( 
 \begin{array}{cccccccccccc}
  \fr{n} & \fr{m}_{01} &\cdots & & \fr{m}_{0s} &   \vline & \fr{n} & \fr{m}_{01} & \cdots& & \fr{m}_{0s} \\ 
 \fr{m}_{01}  & \square & \fr{m}_{12} & \cdots & \fr{m}_{1s} &  \vline & \fr{m}_{01} &\square & \fr{m}_{12} & \cdots & \fr{m}_{1s}  \\
 \fr{m}_{02}  & \fr{m}_{12} & \square &\ddots & \vdots   & \vline  &\fr{m}_{02} &\fr{m}_{12} & \square &\ddots & \vdots \\
 \vdots  & \vdots & &\ddots  & \fr{m}_{s-1,s}  & \vline & \vdots & \vdots& &\ddots  & \fr{m}_{s-1,s} \\
 \fr{m}_{0s}  & \fr{m}_{1s} &  \cdots &  & \square  & \vline & \fr{m}_{02} & \fr{m}_{1s} & \cdots & & \square & \\
     \hline
     \fr{n} & \fr{m}_{01} &\cdots & & \fr{m}_{0s} &   \vline & \fr{n} & \fr{m}_{01} & \cdots& & \fr{m}_{0s} \\ 
 \fr{m}_{01}  & \square & \fr{m}_{12} & \cdots & \fr{m}_{1s} &  \vline & \fr{m}_{01} &\square & \fr{m}_{12} & \cdots & \fr{m}_{1s}  \\
 \fr{m}_{02}  & \fr{m}_{12} & \square &\ddots & \vdots   & \vline  &\fr{m}_{02} &\fr{m}_{12} & \square &\ddots & \vdots \\
 \vdots  & \vdots & &\ddots  & \fr{m}_{s-1,s}  & \vline & \vdots & \vdots& &\ddots  & \fr{m}_{s-1,s} \\
 \fr{m}_{0s}  & \fr{m}_{1s} &  \cdots &  & \square  & \vline & \fr{m}_{02} & \fr{m}_{1s} & \cdots & & \square & 
\end{array}
  \right).  
 \end{equation*}    
   
\begin{prop}\label{SubmoduleBrackets3} The following relations are satisfied by the  modules in the decomposition {\rm (\ref{isotrsp})}:
\begin{equation}\label{Temp0}
[ \fr{sp}(n_0), \fr{sp}(n_0)] \subseteq\fr{sp}(n_0),
\end{equation}
\begin{equation}\label{Temp1}[ \fr{m}_{ij}, \fr{m}_{jk}] \subseteq \fr{m}_{ik}, \ \ 0\leq i<j<k \leq s.
\end{equation}
\begin{equation}\label{temp2} 
[\fr{sp}(n_i),\fr{m}_{lm}]=\left\{ 
\begin{array}{lll}
\fr{m}_{lm},  \quad \makebox{if $i=l$ or $i=m$} \ \\

\{0\} ,   \quad \makebox{otherwise},
\end{array}
\right. \ \ l<m, \ \ i,l,m=0,\dots,s.
\end{equation}
\end{prop}
\begin{proof} The first relation is true since $\fr{sp}(n_0)$ is a subalgebra of $\fr{g}=\fr{sp}(n)$.  The second and the third relations follow from Lemma \ref{rel2} and the expression of the submodules $\fr{m}_{ij}$ in terms of the basis $\mathcal{B}$. \end{proof}

\begin{remark} The above relations are not the only relations among the modules in decomposition {\rm (\ref{isotrsp})} which are valid, however these are the ones which we use in our study.
\end{remark}

\section{Some preliminary results}\label{proofPrep}
In order to prove our main theorem \ref{main2}} we will need some  propositions and lemmas, which we collect in the present section.

\smallskip
The following can be easily proved by induction on $s$.
\begin{lemma}\label{CombinatorialLemma1}
Let $R_s=\{\lambda_{ij}: \ 0\leq i<j\leq s\}$ be a set of real numbers such that $\lambda_{ij}=\lambda_{jk}=\lambda_{ik}$ for all $0\leq i<j<k\leq s$. Then $R_s$ is a  singleton.
\end{lemma}


Consider the space $G/H=Sp(n)/Sp(n_1)\times \cdots \times Sp(n_s)$ and recall the description of its tangent space given in Section \ref{Isotropy2}.  Let $g$ be a $G$-invariant g.o. metric on $G/H$ with corresponding metric endomorphism  $A:\fr{m}\rightarrow \fr{m}$ (cf. Equation \eqref{MetEnd}), where we take $B$ to be the multiple  of the Killing form \eqref{KillSp(n)}.  Recall the spaces $\fr{n}=\fr{sp}(n_0)$ ($n_0=n- (n_1+\cdots +n_s)$) and  $\fr{p}$, which were defined in  decomposition \eqref{dd2}.
 
\begin{prop}\label{Simplif1sp}
$\left.A\right|_{\fr{p}}=\lambda\op{Id}$, where $\lambda>0$. 
\end{prop} 
\begin{proof} 
 By the normalizer Lemma \ref{NormalizerLemma} the metric endomorphism $A$ is $\op{Ad}(N_G(H^0))$-equivariant.  As a result, $A$ is $\op{ad}(\fr{n}_{\fr{g}}(\fr{h}))$-equivariant, where $\fr{n}_{\fr{g}}(\fr{h})=\{Y\in \fr{g}: [Y,\fr{h}]\subseteq \fr{h}\}$ is the Lie algebra of $N_G(H^0)$. In our case $\fr{h}=\fr{sp}(n_1)\oplus \cdots \oplus \fr{sp}(n_s)$ and 
\begin{equation*}
\fr{n}_{\fr{g}}(\fr{h})=\fr{sp}(n_0)\oplus \fr{sp}(n_1)\oplus \cdots \oplus \fr{sp}(n_s).\end{equation*}   

\noindent By taking into account Lemma \ref{rel2} and the expressions of the subspaces $\fr{m}_{0j},\fr{m}_{ij}, \fr{sp}(n_j)$, $j=0,1,\dots,s$ in terms of the basis $\mathcal{B}$, we deduce that the submodules $\fr{m}_{0j},\fr{m}_{ij}$ are $\op{ad}(\fr{n}_{\fr{g}}(\fr{h}))$-invariant and $\op{ad}(\fr{n}_{\fr{g}}(\fr{h}))$-irreducible.    Apart from being $\op{ad}(\fr{n}_{\fr{g}}(\fr{h}))$-irreducible, the $\op{Ad}(N_G(H^0))$-submodules $\fr{m}_{ij}$, $0\leq i<j\leq s$, are pairwise inequivalent. To see this, firstly observe that if $\fr{m}_{ij}$ and $\fr{m}_{lm}$, with $0\leq i<j\leq s$ and $0\leq l<m\leq s$, are two distinct $\op{Ad}(N_G(H^0))$-submodules then there exists at least one index $i_0$ such that 
either\\
 
 \noindent
 1) $i_0=i$ or $i_0=j$ and $i_0\neq l,m$,\ \ \ or \\
  2) $i_0=l$ or $i_0=m$ and $i_0\neq i,j$.\\

 For case 1), taking into account Equation \eqref{temp2} we obtain  $[\fr{sp}(n_{i_0}),\fr{m}_{ij}]=\fr{m}_{ij}$ and $[\fr{sp}(n_{i_0}),\fr{m}_{lm}]=\{0\}$.  For case 2), Equation \eqref{temp2} yields $[\fr{sp}(n_{i_0}),\fr{m}_{ij}]=\{0\}$ and $[\fr{sp}(n_{i_0}),\fr{m}_{lm}]=\fr{m}_{lm}$.  By virtue of Lemma \ref{EquivalentLemma} we deduce that the submodules $\fr{m}_{ij}$ and $\fr{m}_{lm}$ are inequivalent in both cases.
 
  On the other hand, and in view of Remark \ref{conclusion}, the restriction $\left.A\right|_{\fr{p}}=\left.A\right|_{\bigoplus_{0\leq i<j\leq s}\fr{m}_{ij}}$ defines a g.o. metric on $G/N_G(H^0)$.  Since the submodules $\fr{m}_{ij}$ are $\op{Ad}(N_G(H^0))$-irreducible and pairwise inequivalent, we have
\begin{equation}\label{Condition3}\left.A\right|_{\fr{m}_{ij}}=\lambda_{ij}\op{Id},  \ \ 0\leq i<j<l\leq s.\end{equation}

\noindent Also, Equation (\ref{Temp1}) yields
\begin{equation*}\label{SubmoduleBrackets2}
[\fr{m}_{ij},\fr{m}_{jl}]\subseteq \fr{m}_{il} \ \ \makebox{for all} \ \ 0\leq i<j<l\leq s.
\end{equation*}

\noindent In particular, since $[\fr{m}_{ij},\fr{m}_{jl}]$ is an $\op{Ad}(N_G(H^0))$-submodule and $\fr{m}_{il}$ is irreducible, we have $[\fr{m}_{ij},\fr{m}_{jl}]=\fr{m}_{il}$.  Due to the fact that $\left.A\right|_{\fr{p}}$ defines a g.o. metric on $G/N_G(H^0)$, along with Equation \eqref{Condition3}, the $\op{ad}({\fr{h}})$-invariance of $\fr{m}_{ij},\fr{m}_{jl}$ and their $B$-orthogonality, part \textbf{2.} of Lemma \ref{EigenEq} yields 
\begin{equation*}\label{EqEigen1}\lambda_{ij}=\lambda_{jl}=\lambda_{il} \ \ \makebox{for all} \ \ 0\leq i<j< l\leq s.\end{equation*}

 \noindent By Lemma \ref{CombinatorialLemma1}, we deduce that the set $R_s=\{\lambda_{ij}: \ 0\leq i<j\leq s\}$ of the eigenvalues of $\left.A\right|_{\fr{p}}$ contains only one element, and thus $\left.A\right|_{\fr{p}}=\lambda\op{Id}$.
 \end{proof}

  Now assume that $0<n_0\neq 1$.  The algebra $\fr{n}=\fr{sp}(n_0)$ coincides with the Lie algebra of $N_G(H^0)/H^0$.  By Lemma \ref{DualNormalizer}, $\left.A\right|_{\fr{n}}$ defines a bi-invariant metric on $N_G(H^0)/H^0$, which in turn corresponds to an $\op{Ad}$-invariant inner product on $\fr{sp}(n_0)$.  Note that $\fr{sp}(n_0)$ is simple and the only $\op{Ad}$-invariant inner product is a scalar multiple of the Killing form.  Therefore, 
  \begin{equation}\label{mueigenv}\left.A\right|_{\fr{n}}=\mu\op{Id}, \ \ \mu>0.\end{equation}

Next, we have the following:
\begin{prop}\label{Simplif2sp}
We assume that $n_0>1$. Then $\lambda=\mu$, where $\lambda$ is given by Proposition \ref{Simplif1sp} and $\mu$ is given by Equation \eqref{mueigenv}.
\end{prop}
\begin{proof}

Recall the $\op{Ad}(G/H)$-submodules $\fr{m}_l^j\subseteq \fr{m}_{0j}\subset \fr{p}$ defined in Section \ref{Isotropy2}. By Proposition \ref{Simplif1sp}, we have
\begin{equation*}\label{Res2sp}\left.A\right|_{\fr{m}_j^i}=\lambda\op{Id}.
\end{equation*}

\noindent We choose the vectors $e_{12}\in \fr{n}$ and $e_{1k+1}\in \fr{m}^1_1$. Since $n_0>1$, the above vectors do not coincide, hence  by Lemma \ref{rel2} it follows that
\begin{equation*}
[e_{12},e_{1k+1}]=-e_{2k+1}\in\fr{m}^1_2.
\end{equation*}

\noindent Therefore, $[\fr{n},\fr{m}_1^1]$ has non zero projection on $(\fr{n}\oplus \fr{m}_1^1)^{\bot}$.  Along with the $\op{ad}(\fr{h})$-invariance of $\fr{n}$ and $\fr{m}_1^1$ and the facts that $\left.A\right|_{\fr{m}_j^i}=\lambda\op{Id}$ and $\left.A\right|_{\fr{n}}=\mu\op{Id}$, part \textbf{1.} of Lemma \ref{EigenEq} yields $\lambda=\mu$.   
\end{proof}

  Finally, we consider the case $n_0=1$.
  
  \begin{prop}\label{Simplif3sp} Let $n_0=1$ and $s>1$. Then $\lambda=\mu$, where $\lambda$ is given by Proposition \ref{Simplif1sp} and $\mu$ is given by Equation \eqref{mueigenv}.
  \end{prop}

 The proof  requires the following.

\begin{lemma}\label{finalemma}
Let $\fr{m}_{ij}$, $i>0$, be one of the submodules defined in Section \ref{Isotropy2}.  Let $a\in \fr{h}=\fr{sp}(n_1)\oplus \cdots \oplus \fr{sp}(n_s)$ such that $[a,\fr{m}_{ij}]=\{0\}$.  Then the projection of $a$ on $\fr{sp}(n_i)\oplus \fr{sp}(n_j)$ is zero, i.e. $a\in \fr{sp}(n_1)\oplus \cdots \oplus \fr{sp}(n_{i-1})\oplus \fr{sp}(n_{i+1})\oplus \cdots \oplus \fr{sp}(n_{j-1})\oplus \fr{sp}(n_{j+1})\oplus \cdots \oplus \fr{sp}(n_s)$.
\end{lemma}
 \begin{proof}  Let $\pi_{j}(a)$ be the projection of $a$ on $\fr{sp}(n_j)$ and set $\fr{h}_{ij}:=\fr{sp}(n_i)\oplus \fr{sp}(n_j)$.  By relation \eqref{temp2}, we deduce that $[\fr{h},\fr{m}_{ij}]=[\fr{h}_{ij},\fr{m}_{ij}]=\fr{m}_{ij}$.  Then we have that $\{0\}=[a,\fr{m}_{ij}]=[(\pi_{i}+\pi_j)(a),\fr{m}_{ij}]$, therefore $(\pi_{i}+\pi_j)(a)$ lies in the space 
 \begin{equation*}c_{ij}:=\{X\in \fr{h}_{ij}:[X,Y]=0 \ \ \makebox{for all} \ \ Y\in \fr{m}_{ij}\}.\end{equation*}
 
 \noindent Using the Jacobi identity and the $\op{ad}(\fr{h}_{ij})$-invariance of $\fr{m}_{ij}$, it is not hard to verify that the space $c_{ij}$ is an ideal of the Lie algebra $\fr{h}_{ij}$.  Since $\fr{h}_{ij}$ is semisimple, $c_{ij}$ is necessarily one of the ideals $\{0\}$, $\fr{sp}(n_i)$, $\fr{sp}(n_j)$ or $\fr{h}_{ij}$.  On the other hand, relation \eqref{temp2} implies that $[\fr{sp}(n_i),\fr{m}_{ij}]\neq \{0\}$, $[\fr{sp}(n_j),\fr{m}_{ij}]\neq \{0\}$ and $[\fr{sp}(n_i)\oplus \fr{sp}(n_j),\fr{m}_{ij}]\neq \{0\}$.  We conclude that $c_{ij}=\{0\}$ and thus $(\pi_{i}+\pi_j)(a)=0$. \end{proof}

\noindent \emph{Proof of Proposition \ref{Simplif3sp}}. In view of Remark \ref{conclusion}, Proposition \ref{Simplif1sp} and Equation \eqref{mueigenv}, the g.o. metric on $G/H$ has the form 
\begin{equation}\label{g.o.form3}A=\begin{pmatrix}\left.\mu \op{Id}\right|_{\fr{n}}& 0& \\ 0& \left.\lambda\op{Id}\right|_{\fr{p}}\end{pmatrix}.
\end{equation}

\noindent Let $X\in \fr{m}=\fr{n}\oplus \fr{p}$ and write $X=X_{\fr{n}}+X_{\fr{p}}$, where $X_{\fr{n}}$ is the projection of $X$ on $\fr{n}$ and $X_{\fr{p}}$ is the projection of $X$ on $\fr{p}$.  Further, we write 
\begin{equation}\label{Xp}X_{\fr{p}}=\sum_{0\leq i<j\leq s}X_{ij},\end{equation} 

\noindent where $X_{ij}$ denotes the projection of $X$ on the space $\fr{m}_{ij}$.  Taking into account the definition of the spaces $\fr{n}$, $\fr{m}_{ij}$ in terms of the basis $\mathcal{B}$, along with Lemma \ref{rel2}, we deduce that
\begin{equation}\label{nrel}[\fr{n},\fr{m}_{ij}]\subseteq \left\{ \begin{array}{ll}  \fr{m}_{ij},\ \ \mbox{if}\ \ i=0 \\
\left\{ {0} \right\}, \quad \mbox{otherwise}.
\end{array}
\right.
\end{equation}

\noindent Since $A$ is a g.o. metric, Proposition \ref{GOCond} along with expression (\ref{g.o.form3}) of $A$, implies that there exists a vector $a\in \fr{h}=\fr{sp}(n_1)\oplus \cdots \oplus \fr{sp}(n_s)$ such that 
\begin{equation*}0=[a+X_{\fr{n}}+X_{\fr{p}}, \mu X_{\fr{n}}+\lambda X_{\fr{p}}]=\mu[a,X_{\fr{n}}]+\lambda [a,X_{\fr{p}}]+(\lambda-\mu)[X_{\fr{n}},X_{\fr{p}}].\end{equation*}

Taking into account Equation \eqref{Xp} along with relation \eqref{nrel} and the fact that $[\fr{n},\fr{h}]=\{0\}$ (recall that $\fr{n}=\fr{sp}(n_0)$), the above equation is equivalent to 
\begin{equation}\label{eqgo}0=\lambda\sum_{0\leq i<j\leq s}[a,X_{ij}]+(\lambda-\mu)\sum_{j=1}^s[X_{\fr{n}},X_{0j}].
\end{equation}

\noindent We choose $X_{\fr{n}}:=f_{11}$, $X_{01}:=e_{12}$ and $X_{\fr{p}}:=X_{01}+X_{12}=e_{12}+X_{12}\in \fr{m}_{01}+\fr{m}_{12}$ (since $s>1$, $\fr{m}_{12}\neq \{0\}$), where we consider $X_{12}$ as arbitrary. Then Equation \eqref{eqgo} along with Lemma \ref{rel2} yield 
\begin{equation}\label{eqgo1}0=\lambda[a,e_{12}]+\lambda [a, X_{12}]+2(\lambda-\mu)f_{21}.\end{equation}

  \noindent By the $\op{ad}(\fr{h})$-invariance of $\fr{m}_{ij}$ and the fact that $a\in \fr{h}$, the first two terms of Equation \eqref{eqgo1} lie in $\fr{m}_{01}$ and $\fr{m}_{12}$ respectively, while the last term lies in $\fr{m}_{01}$. Therefore, Equation \eqref{eqgo1} yields the system
  \begin{eqnarray} \lambda[a,e_{12}]+ 2(\lambda-\mu)f_{21}&=&0 \label{e1}\\
   \lambda [a, X_{12}]&=&0.\label{e2}
   \end{eqnarray}
   
   \noindent Since $X_{12}\in \fr{m}_{12}$ is arbitrary, Equation \eqref{e2} implies that $[a,\fr{m}_{12}]=\{0\}$.  Lemma \ref{finalemma} then yields $(\pi_{1}+\pi_2)(a)=0$, where $\pi_j$ denotes the projection of $a$ on $\fr{sp}(n_j)$.  Hence $\pi_1(a)=\pi_2(a)=0$.  On the other hand, the fact that $e_{12}\in \fr{m}_{01}$ along with relation \eqref{temp2}, imply that $[a,e_{12}]=[\pi_{1}(a),e_{12}]=0$.  Substituting into Equation \eqref{e1}, we obtain $2(\lambda-\mu)f_{21}=0$ and thus $\lambda=\mu$. \qed

\section{Proof of Theorem \ref{main2}}\label{proof2}
 
We can now combine the results of the previous section to give a proof of Theorem \ref{main2}.

 The standard metric on $G/H=Sp(n)/Sp(n_1)\times \cdots \times \Sp(n_s)$ is a g.o. metric, hence the sufficiency part of the theorem holds trivially. 
  For the necessity part, let $g$ be a $G$-invariant g.o. metric on $G/H$. If $n_0=0$, i.e.  $\fr{n}=\{0\}$, then Proposition \ref{Simplif1sp} implies that $g$ is the standard metric.  

  Now assume that $n_0>0$.  
  If $n_0\ne 1$, then the theorem follows from Remark \ref{conclusion} and Propositions \ref{Simplif1sp} and \ref{Simplif2sp}.
If $n_0=1$ then $\fr{n}=\fr{sp}(1)$ and we consider two cases, $s=1$ and $s>1$. 

If $s=1$, then $G/H=Sp(n)/Sp(n-1)$ and in this case any metric endomorphism $A$ such that $\left.A\right|_{\fr{p}}=\lambda \op{Id}$ and $\left.A\right|_{\fr{n}}=\mu\op{Id}$ is a g.o. metric and vice versa (\cite{Nik0}).   Therefore, the $G$-invariant g.o metrics on $G/H$ are precisely the metrics of the form $A=\begin{pmatrix}\left.\mu \op{Id}\right|_{\fr{n}}& 0& \\ 0& \left.\lambda\op{Id}\right|_{\fr{p}}\end{pmatrix}$.  It follows that the $G$-invariant g.o. metrics on $G/H$ are homothetic to the metrics $g_{\mu}=\begin{pmatrix}\left.\mu \op{Id}\right|_{\fr{n}}& 0& \\ 0& \left.\op{Id}\right|_{\fr{p}}\end{pmatrix}$. This settles Theorem \ref{main2} for the case $n_0=1$ and $s=1$.
Finally, if $n_0=1$ and $s>1$, Theorem \ref{main2} follows from Proposition  \ref{Simplif3sp}.\qed


\end{document}